\documentclass[11pt]{article}

\usepackage{amssymb,amsbsy,amsthm,amsmath,graphicx,epsfig}
\usepackage{mathabx,times}

\usepackage{sectsty}

\sectionfont{\center\normalsize}

\subsectionfont{\normalsize}

\begin{document}


\newtheorem{thm}{Theorem}[section]
\newtheorem{prop}[thm]{Proposition}
\newtheorem{lem}[thm]{Lemma}
\newtheorem{dfn}[thm]{Definition}
\newtheorem{cor}[thm]{Corollary}
\newtheorem{ques}[thm]{Question}
\theoremstyle{remark}
\newtheorem*{rmk}{Remark}
\newtheorem{ex}[thm]{Example}


\renewcommand{\to}{\longrightarrow}
\newcommand{\actson}{\curvearrowright}

\newcommand{\Con}{\mathrm{Con}_\infty}

\renewcommand{\a}{\alpha}
\renewcommand{\b}{\beta}
\renewcommand{\d}{\mathrm{d}}
\newcommand{\eps}{\varepsilon}
\newcommand{\g}{\gamma}
\renewcommand{\k}{\kappa}
\newcommand{\s}{\sigma}
\newcommand{\G}{\Gamma}
\renewcommand{\L}{\Lambda}

\renewcommand{\phi}{\varphi}

\newcommand{\bbZ}{\mathbb{Z}}
\newcommand{\bbN}{\mathbb{N}}
\newcommand{\bbR}{\mathbb{R}}

\renewcommand{\cal}[1]{\mathcal{#1}}
\renewcommand{\rm}[1]{\mathrm{#1}}
\renewcommand{\bf}[1]{\mathbf{#1}}
\renewcommand{\t}[1]{\tilde{#1}}
\newcommand{\ol}[1]{\overline{#1}}

\renewcommand{\qed}{\hspace{\stretch{1}}$\Box$}
\newcommand{\fin}{\hspace{\stretch{1}}$\lhd$}

\title{\textbf{\Large{Integrable measure equivalence for groups of polynomial growth}}}

\author{by Tim Austin\footnote{supported by a fellowship from the Clay Mathematics Institute}\\ \\ \emph{With an Appendix by Lewis Bowen}}

\date{}

\maketitle

\begin{abstract}
Bader, Furman and Sauer have recently introduced the notion of integrable measure equivalence for finitely-generated groups.  This is the sub-equivalence relation of measure equivalence obtained by insisting that the relevant cocycles satisfy an integrability condition.  They have used it to prove new classification results for hyperbolic groups.

The present work shows that groups of polynomial growth are also quite rigid under integrable measure equivalence, in that if two such groups are equivalent then they must have bi-Lipschitz asymptotic cones.  This will follow by proving that the cocycles arising from an integrable measure equivalence converge under re-scaling, albeit in a very weak sense, to bi-Lipschitz maps of asymptotic cones.
\end{abstract}

\section{Introduction}
Measure equivalence is an equivalence relation on groups introduced by Gromov in~\cite{Gro93}.  It has since become the object of considerable study: Furman's survey~\cite{Furman11} provides a thorough overview.  However, it is essentially trivial for countably infinite amenable groups.  This is because two groups are measure equivalent whenever they have free orbit-equivalent probability-preserving ergodic actions.  Such actions exist for any infinite group, since Bernoulli shifts give examples, and Ornstein and Weiss proved in~\cite{OrnWei80} that any two such actions of any countably infinite amenable groups are orbit-equivalent, generalizing the classical theorems of Dye~\cite{Dye59,Dye63} about $\bbZ$-actions.

A measure equivalence between two groups implicity defines a pair of (equivalence classes of) cocycles over probability-preserving actions of those groups.  In~\cite{BadFurSau--IME}, Bader, Furman and Sauer have sharpened measure equivalence to a finer equivalence relation by allowing only measure equivalences for which these cocycles satisfy an integrability condition.  This sharper relation is called integrable measure equivalence, henceforth abbreviated to IME.

Their focus is on applications to rigidity of hyperbolic lattices.  The present paper considers instead finitely-generated groups of polynomial growth, and finds that these also exhibit considerable rigidity for IME, in sharp contrast to the original notion of measure equivalence.  The rigidity for these `small' groups is in terms of Gromov's notion of their asymptotic cones.

\begin{thm}\label{thm:main}
If $G$ and $H$ are f.-g. groups of polynomial growth which are IME, then there is a bi-Lipschitz bijection $\Con G\to \Con H$ between their asymptotic cones.
\end{thm}

Here the notation `$\Con G$' refers to the asymptotic cone of a group $G$ with a given right-invariant word metric $d_G$, as constructed in \cite[Chapter 2]{Gro93}.  By Gromov's Theorem in~\cite{Gro81} that f.-g. groups of polynomial growth are virtually nilpotent, Theorem~\ref{thm:main} is effectively a theorem about nilpotent groups.  For general groups, the construction of $\Con G$ may depend on the choice of a non-principal ultrafilter~(\cite{Tho(S)Vel00}), but for nilpotent groups, and hence groups of polynomial growth, it is known to be independent of that choice.  (We will later invoke more precise results of Pansu which imply this.)

One can see Theorem~\ref{thm:main} as a generalization to polynomial-growth groups of the result that an integrable measure equivalence between $\bbZ^d$ and $\bbZ^D$ must asymptotically define an isomorphism $\bbR^d\to \bbR^D$, and hence requires that $d = D$.  This special case follows easily by applying the Norm Ergodic Theorem to the cocycles defining the measure equivalence.

In the setting of more general groups, Lewis Bowen has shown that the growth function of a f.-g. group is an IME-invariant.  His exposition is given as a self-contained appendix to the present paper.  That result already implies that the amenable groups fall into many (indeed, uncountably many) distinct IME-classes, and that the subclass of groups of polynomial growth is IME-closed.  However, it seems that more subtle arguments are needed, for example, to distinguish the discrete Heisenberg group from $\bbZ^4$ up to IME, since both of these groups having quartic growth.  Theorem~\ref{thm:main} implies that they are not IME, because
\begin{multline*}
\Con(\hbox{discrete Heis},\hbox{word metric})\\ \cong_{\scriptsize{\hbox{bi-Lip}}} (\hbox{continuous Heis},\hbox{Carnot-Carath\'eodory metric})
\end{multline*}
and
\[\Con \bbZ^4 \cong_{\scriptsize{\hbox{bi-Lip}}} \bbR^4,\]
and these are not bi-Lipschitz (for instance, because their dimensions as topological spaces do not match).

More generally, Bowen's result implies that if $G$ is IME to $\bbZ^d$ then $G$ must be of polynomial growth, and then Theorem~\ref{thm:main} implies that $\Con G \cong_{\scriptsize{\hbox{bi-Lip}}} \bbR^d$.  It is known that $\Con G$ is always a graded connected nilpotent Lie group, and it is a Euclidean space only if $G$ was virtually Abelian~(\cite[Chapter 2]{Gro93}), so our remarks about the Heisenberg group generalize to the following.

\begin{cor}
If a f.-g. group $G$ is IME to $\bbZ^d$ for some $d$, then $G$ is virtually $\bbZ^d$. \qed
\end{cor}

We will also need the invariance of the growth function for an auxiliary purpose during our proofs later.

For nilpotent groups, the map
\[G \mapsto \Con G\]
seems to retain a great deal of large-scale geometric information about $G$.  The main result of Pansu's work~\cite{Pan89} is a precise characterization of those pairs of f.-g. nilpotent groups whose asymptotic cones are bi-Lipschitz: this is equivalent to isomorphism of their associated graded Lie algebras.  Moreover, for Carnot groups (that is, nilpotent groups which admit an endomorphism which enlarges all distances by a fixed factor), such as any $\bbZ^d$ or the Heisenberg group, it is known that $G$ is quasi-isometric to $\Con G$ (see example 2.C$_1$(a) in~\cite{Gro93}).  For other f.-g. nilpotent groups, the issue of just what geometric information is retained by the construction of the asymptotic cone is still not completely understood.

\subsection*{Acknowledgements}

This work emerged from an ongoing collaboration with Uri Bader, Lewis Bowen, Alex Furman, and Romain Sauer.  I am grateful to them for their enthusiasm and for numerous productive discussions.

Yves de Cornulier gave helpful feedback on an earlier version of this paper, including bringing his work~\cite{deC11} to my attention.

Part of this work was completed during a visit to the Korea Institute for Advanced Study.

\section{Background and first steps}

\subsection{Integrable measure equivalence}\label{subs:basics}

This paper will largely assume the basic facts about measure equivalence and integrable measure equivalence: we will recall only a brief statement of them here.  We essentially follow the treatment in Furman's survey~\cite{Furman11} (which is also similar to Section 1.2 and Appendix A of~\cite{BadFurSau--IME}).

Given countable discrete groups $G$ and $H$, a \textbf{measure coupling} between them is a nonzero $\s$-finite measure space $(\Omega,m)$ which admits commuting $m$-preserving actions of $G$ and $H$ which both have finite-measure fundamental domains.  We denote the actions of both $G$ and $H$ on $\Omega$ by $\cdot$.  By restricting attention to an ergodic component, one may always assume that $m$ is ergodic for the resulting $G\times H$-action on $\Omega$. The fundamental domains $Y$ and $X$ for the $G$- and $H$-actions give rise to functions $\b:H\times Y \to G$ and $\a:G\times X\to H$, defined uniquely by requiring that
\[h\cdot y \in \b(h,y)^{-1}\cdot Y \quad \hbox{and} \quad g\cdot x \in \a(g,x)^{-1}\cdot X \quad \forall x \in X,\ y \in Y\]
(the inverses are inserted so that some other calculations come out simpler later).  This also defines auxiliary finite-measure-preserving actions $S:H\actson (Y,m|_Y)$ and $T:G\actson (X,m|_X)$ by requiring that
\[h\cdot y = \b(h,y)^{-1}\cdot (S^hy) \quad \hbox{and} \quad g\cdot x = \a(g,x)^{-1}\cdot (T^gx).\]
If $m$ is ergodic for $G\times H$, then $m|_X$ is ergodic for $T$ and $m|_Y$ is ergodic for $S$.  These are both finite invariant measures, but at times it will be convenient to insist on probability measures: for those situations, we will set
\[\mu_X := m(X)^{-1}\cdot m|_X \quad \hbox{and} \quad \mu_Y := m(Y)^{-1}\cdot m|_Y.\]

Now a standard calculation shows that $\b$ and $\a$ are cocycles over $S$ and $T$ respectively: that is,
\[\a(g_1g_2,x) = \a(g_1,T^{g_2}x)\a(g_2,x) \quad \forall g_1,g_2 \in G,\ x \in X,\]
and similarly for $\b$.

In this construction, we may always replace the fundamental domain $Y$ with one of its $H$-translates, and the cocycle $\b$ will just be translated accordingly.  Since countably many translates of $Y$ cover $\Omega$, we may therefore ensure that $m(X \cap Y) > 0$.  Now a simple calculation shows that if
\[x \in X\cap Y \cap T^{g^{-1}}(X \cap Y) \quad \hbox{for some}\ g \in G,\]
then we may write
\[g^{-1}\cdot (T^gx) = \a(g^{-1},T^gx)^{-1}\cdot x = \b(\a(g^{-1},T^gx)^{-1},x)^{-1}\cdot (S^{\a(g^{-1},T^gx)^{-1}}x),\]
where the first equality holds because $T^gx \in X$, and the second because $x \in Y$.  Since we also assume that $T^gx \in Y$, and the $G$-translates of $Y$ are disjoint, this implies that
\[\b(\a(g^{-1},T^gx)^{-1},x) = g \quad \hbox{and} \quad S^{\a(g^{-1},T^gx)^{-1}}x = T^g x.\]
Finally, the cocycle equation for $\a$ gives that $\a(g^{-1},T^gx) = \a(g,x)^{-1}$, so these conclusions simplify to
\begin{eqnarray}\label{eq:inv-reln}
\b(\a(g,x),x) = g \quad \hbox{and} \quad S^{\a(g,x)}x = T^g x.
\end{eqnarray}

In particular, the orbit equivalence relations of $T$ on $X$ and $S$ on $Y$ have the same restriction to $X\cap Y$.

In the sequel, it will often be convenient to work instead with the functions $\a_x := \a(\,\cdot\,,x):G\to H$ and $\b_y := \b(\,\cdot\,,y):H\to G$.  The cocycle equation for $\a$ gives that $x\mapsto \a_x$ is a map from $X$ to
\[[G,H] := \{f:G\to H\,|\ f(e_G) = e_H\}\]
which intertwines the action $T:G\actson X$ with the action of $G$ on $[G,H]$ defined by $g:f(x) \mapsto f(xg)f(g)^{-1}$.  Similarly, $\b$ is a map from $Y$ to $[H,G]$ which intertwines $S$ with the analagous action of $H$ on $[H,G]$.  With this interpretation, the pushforward of $\mu_X$ under $x\mapsto \a_x$ is an invariant probability on $[G,H]$: such objects are discussed by Monod in~\cite{Mon06} under the term `randomorphisms', and again by Furman~\cite[Subsection 2.3]{Furman11}.  (Also, in the special case of $[\bbZ^2,\bbZ]$, they have a long history in statistical physics as models of random surfaces: see, for instance,~\cite{Sheff05} and the many references there.)

Now, for $x \in X$ and $y \in Y$, let
\[D_x := \{g \in G\,|\ T^gx \in X\cap Y\}\]
and
\[E_y := \{h \in H\,|\ S^hy \in X\cap Y\}.\]
Then $x\mapsto D_x$ is a map
\[X\to \{\hbox{subsets of}\ G\}\]
which is equivariant in the sense that
\begin{eqnarray}\label{eq:D-equivt}
D_{T^gx} = D_x\cdot g^{-1},
\end{eqnarray}
and similarly for $y\mapsto E_y$.

Also, if $m$ is ergodic for $G\times H$, then $m|_Y$ is ergodic for $S$ and $m|_X$ is ergodic for $T$.  Using this, we may extend the definitions of $D_\bullet$ and $E_\bullet$, $\a_\bullet$ and $\b_\bullet$ to almost all of $X\cup Y$.  By ergodicity, for $m$-a.e. $y \in Y$ the set $E_y$ is nonempty, so there is some $h \in H$ such that $S^hy \in X\cap Y$.  This now gives
\[y = S^{h^{-1}}S^hy = T^{\b(h^{-1},S^hy)}S^hy = T^{\b(h,y)^{-1}}S^hy,\]
using~(\ref{eq:inv-reln}) and the cocycle equation for $\b$.  Setting
\[D_y := D_{S^hy}\cdot \b(h,y),\]
this is independent from the choice of $h$ by the cocycle relations.  Similarly, for $m$-a.e. $x \in X$ there is $g \in G$ such that $T^g x \in X\cap Y$, and now we may set
\[E_x := E_{T^gx}\cdot \a(g,x).\]
For the cocycles, if $y \in Y$ and $h$ is chosen as above, we set
\[\a_y(g) := \a_{S^hy}(g\b(h,y)^{-1})\a_{S^hy}(\b(h,y)^{-1})^{-1},\]
and similarly
\[\b_x(h) := \b_{T^gx}(h\a(g,x)^{-1})\b_{T^gx}(\a(g,x)^{-1})^{-1}.\]
Once again, the consistency of these definitions follows from the cocycle relations for $\a$ and $\b$.

Having thus extended these objects, the relation~(\ref{eq:inv-reln}) now asserts that $\a_x|D_x$ is a bijection $D_x\to E_x$ for every $x \in X\cup Y$, and its inverse equals $\b_x|E_x$.

Our subsequent reasoning about measure equivalence will mostly be in terms of these equivariant maps $x\mapsto (\a_x,D_x)$ and $y\mapsto (\b_y,E_y)$.

For any f.-g. groups $G$ and $H$ and a probability-preserving action $T:G\actson (X,\mu)$, a cocycle $\a:G\times X\to H$ is \textbf{integrable} if, for any choice of finite, symmetric generating set $B_H\subseteq H$, we have
\[\||\a(g,\,\cdot\,)|_H\|_1 = \int_X|\a(g,x)|_H\,\mu(\d x) < \infty \quad \forall g\in G,\]
where $|\cdot|_H$ is the length function on $H$ associated to $B_H$.  Since the length functions arising from different choices of $B_H$ are all equivalent up to constants, this notion does not depend on the choice of $B_H$.  Moreover, the subadditivity of $|\cdot|_H$ gives
\begin{eqnarray}\label{eq:cocyc-subadd}
\||\a(g,\,\cdot\,)|_H\|_1 \leq |g|_G\cdot \max_{s \in B_G}\||\a(s,\,\cdot\,)|_H\|_1,
\end{eqnarray}
where $B_G$ is a finite, symmetric generating set for $G$, so it suffices to check integrability on this $B_G$.

A measure coupling as above is \textbf{integrable} if one can choose fundamental domains $X$ and $Y$ so that the cocycles $\a$ and $\b$ are integrable.

Finally, f.-g. groups $G$ and $H$ are \textbf{integrably measure equivalent}, or \textbf{IME}, if they admit an integrable measure coupling.  Standard arguments, given in~\cite{Furman11}, show that this defines an equivalence relation on f.-g. groups, independent of the choice of their generating sets.  It will be denoted by $\stackrel{\rm{IME}}{\sim}$.

\subsection{Initial simplification}

In our setting, standard properties of IME lead to an immediate, useful reduction of the task of proving Theorem~\ref{thm:main}.  According to Gromov's famous result from~\cite{Gro81}, any f.-g. group $G$ of polynomial growth has a f.-g. nilpotent subgroup $G_1$ of finite index.  Letting $\Omega_1:= G$ with counting measure, this defines a $(G_1,G)$-coupling
\[(g_1,g)\cdot \omega := g_1\omega g^{-1}.\]
Since $G_1$ has a finite fundamental domain in $G$, this measure coupling is trivially integrable.  The same reasoning holds for some finite-index nilpotent subgroup $H_1 \leq H$, giving an integrable $(H,H_1)$-measure equivalence.  Therefore, in the setting of Theorem~\ref{thm:main}, we obtain
\[G_1 \stackrel{\rm{IME}}{\sim} G \stackrel{\rm{IME}}{\sim} H \stackrel{\rm{IME}}{\sim} H_1,\]
and hence $G_1 \stackrel{\rm{IME}}{\sim} H_1$, by transitivity.

On the other hand, since asymptotic cones are insensitive to passage to finite-index subgroups, we have
\[\Con G = \Con G_1 \quad \hbox{and} \quad \Con H = \Con H_1.\]

It therefore suffices to prove Theorem~\ref{thm:main} for the subgroups $G_1$ and $H_1$; equivalently, in the special case with $G$ and $H$ themselves nilpotent.  This will simplify some calculations later.

\subsection{Asymptotic cones of nilpotent groups}

Henceforth $G$ and $H$ will be f.-g. nilpotent groups and $B_G$ and $B_H$ will be finite, symmetric generating sets for them.  To the generating set $B_G$ we associate the word-length function $|\cdot|_G$ and the right-invariant word metric $d_G$, and similarly for $B_H$.

It is known that all such groups $G$ with right-invariant word metrics $d_G$ have the following properties:
\begin{enumerate}
\item the asymptotic cone does not depend on the choice of ultrafilter $\omega$ up to pointed isometry, and so may be written as $\Con G$;

\item the sequence of re-scaled pointed metric spaces $(G,e_G,n^{-1}d_G)$ converges as $n\to\infty$ in the local Gromov-Hausdorff sense to the pointed metric space $(\Con G,\ol{e}_G,d_G^\infty)$ for some limit metric $d_G^\infty$ on $\Con G$ (whereas for many groups $\Con$ exists only as an ultralimit);

\item the asymptotic cone $\Con G$ is a proper metric space under $d_G^\infty$ (that is, all bounded sets are precompact).
\end{enumerate}
An element of $\Con G$ will be signified by an overline, as in `$\bar{g}$'.

Most of these properties follow from Pansu's results in~\cite{Pan83}; the last already follows from the theory in~\cite{Gro81}.  For the first, Pansu asserts only independence of the cone from $\omega$ up to a pointed bi-Lipschitz map, but this is tightened to a pointed isometry in~\cite{Bre07}.  On the other hand, in his discussion of asymptotic cones in Chapter 2 of~\cite{Gro93}, Gromov analyses more general groups for which these properties may fail, including (at least for the second property) some solvable examples.

Theorem~\ref{thm:main} will be deduced from the following.

\begin{thm}\label{thm:main2}
If $G$ and $H$ are f.-g. nilpotent groups and $G \stackrel{\rm{IME}}{\sim} H$, then there is a constant $L > 0$ for which the following holds.  For every $R > 0$, there are a finite subset $E \subseteq \Con G$ containing $\ol{e}_G$ and a map $\phi:E\to \Con H$ with the following properties:
\begin{itemize}
\item $\phi(\ol{e}_G) = \ol{e}_H$;
\item $E$ is $(1/R)$-dense in $B^\infty_G(R)$ for the metric $d_G^\infty$;
\item $\phi$ is injective, and $\phi$ and $\phi^{-1}:\phi(E)\to E$ are both $L$-Lipschitz for the limit metrics $d_G^\infty$ and $d_H^\infty$;
\item $\phi(E)$ is $(L/R)$-dense in $B^\infty_H(R/L)$ for the metric $d_H^\infty$.
\end{itemize}
\end{thm}

\begin{proof}[Proof of Theorem~\ref{thm:main} from Theorem~\ref{thm:main2}]
For each $R \in \bbN$, let $E_R$ and $\phi_R$ be a set and map as provided by Theorem~\ref{thm:main2}, and let
\[\G_R := \{(\ol{g},\phi_R(\ol{g}))\,|\ \ol{g} \in E_R\},\]
a finite subset of $\Con G\times \Con H$ which contains the point $(\ol{e}_G,\ol{e}_H)$.

Since $\Con G$ and $\Con H$ are both proper, a diagonal argument gives a subsequence $R_1 < R_2 < \ldots$ such that the intersections $\G_{R_i}\cap (B^\infty_G(r)\times B^\infty_H(r))$ converge in the Hausdorff topology as $i\to\infty$ for every $r \in \bbN$.  This implies that there is a well-defined closed set $\G \subseteq \Con G\times \Con H$ such that
\[\G_{R_i}\cap (B^\infty_G(r)\times B^\infty_H(r)) \to \G\cap (B^\infty_G(r)\times B^\infty_H(r)) \quad \forall r > 0.\]
An easy check shows that that this $\G$ must satisfy
\[\frac{1}{L}d_G^\infty(\ol{g},\ol{g}') \leq d_H^\infty(\ol{h},\ol{h}') \leq Ld_G^\infty(\ol{g},\ol{g}') \quad \forall (\ol{g},\ol{h}),(\ol{g}',\ol{h}') \in \G,\]
so it is the graph an $L$-bi-Lipschitz function between some subsets of $\Con G$ and $\Con H$.  It also sends $\ol{e}_G$ to $\ol{e}_H$.

To finish, we must show that this function has domain the whole of $\Con G$ and image the whole of $\Con H$.  We will prove the latter fact, the former being similar.  For any $\ol{h} \in \Con H$, the fourth assumed property of the sets $\phi_{R_i}(E_{R_i})$ promises a sequence $\ol{g}_i \in E_{R_i}$ such that $\phi_{R_i}(\ol{g}_i) \to \ol{h}$ as $i \to\infty$.  Since every $\phi^{-1}_{R_i}$ is $L$-Lipschitz and maps $\ol{e}_H$ to $\ol{e}_G$, we must have $\ol{g}_i \in B^\infty_G(Ld^\infty_H(\ol{e}_H,\ol{h}))$ for all $i$.  This closed ball is compact, so after passing to a further subsequence we may assume that
\[(\ol{g}_i,\phi_{R_i}(\ol{g}_i)) \to (\ol{g},\ol{h}) \quad \hbox{as}\ i\to\infty\]
for some $\ol{g} \in \Con G$.  This now implies that $(\ol{g},\ol{h}) \in \G$, so $\G$ is the graph of a function onto the whole of $\Con H$.
\end{proof}

\subsection{Invariance of growth}

Our approach to Theorem~\ref{thm:main} will make use of the fact that the growth rate of a f.-g. group is an IME-invariant.  This follows from a more general control of growth functions under `integrable measure embeddings', proved by Lewis Bowen in his appendix to the present paper (Theorem~\ref{thm:growth2}).  The consequence that we will need is as follows.

\begin{lem}\label{lem:growth-est}
If $G$ and $H$ are f.-g. polynomial-growth groups with word metrics $d_G$ and $d_H$ and $G \stackrel{\rm{IME}}{\sim} H$, then for any $M > 0$ there is a constant $D \geq 1$ such that
\[D^{-1}|B_G(D^{-1}Mr)| \leq |B_H(r)| \leq D|B_G(DMr)| \quad \forall r > 0.\]
\end{lem}

\begin{proof}
In case $M = 1$, Bowen's result gives this for arbitrary f.-g. groups.  For nilpotent groups, the case of general $M$ follows because the polynomial growth of those groups implies that the metrics $d_G$ and $d_H$ are doubling.
\end{proof}

\section{A refined growth estimate for cocycles}

If $G$ and $H$ are f.-g. groups with word metrics $d_G$ and $d_H$, $(X,\mu,T)$ is a probability $G$-space and $\s:G\times X\to H$ is an integrable cocycle, then the cocycle identity and an induction on word-length imply that
\[\||\s(g,\cdot)|_H\|_1 \leq C|g|_G\]
for some fixed constant $C$, which may be taken to be $\max_{s \in B_G}\||\s(s,\cdot)|_H\|_1$.  Using Markov's Inequality, this implies that
\[\mu\{|\s(g,x)|_H \geq MC|g|_G\} \leq 1/M \quad \forall M > 0.\]

A key tool in proving Theorem~\ref{thm:main} will be a small but crucial improvement on this estimate in the setting of nilpotent groups.  This is most cleanly formulated in terms of the following abstract notion.

\begin{dfn}
Given any l.c.s.c. group $G$ and probability $G$-space $(X,\mu,T)$, a \textbf{sub-cocycle} over this $G$-space is a measurable function $f:G\times X\to [0,\infty)$ such that
\[f(gh,x) \leq f(g,T^hx) + f(h,x) \quad \hbox{for a.e.}\ x,\ \forall g,h \in G.\]
It is \textbf{integrable} if $f(g,\cdot)$ is integrable for every $g$.
\end{dfn}

This nomeclature is not completely standard.  Setting
\[\rho_x(g,h) := f(gh^{-1},T^hx),\]
one can check that $x \mapsto \rho_x$ is an equivariant map from $(X,T)$ to the space of pseudometrics on $G$ with the action of $G$ given by translation on the right (in particular, the sub-cocycle inequality becomes the triangle inequality).  As with `randomorphisms', important examples of such stationary random pseudometrics for $G = \bbZ^d$ are classical objects in probability: in the study of first-passage percolation models, the first passage times between pairs of points define such a pseudometric.  Classic references for the asymptotic behaviour of this pseudometric include~\cite{HamWel65,CoxDur81,Boi90}, and a recent survey of this area can be found in~\cite{GriKes12}.

In a sense, the next proposition can be seen as very weak nilpotent-groups extension of the convergence of the reachable sets to the limit shape (that is, of these random pseudometrics to a deterministic limiting norm) in first-passage percolation.

\begin{prop}\label{prop:as}
If $G$ is a f.-g. nilpotent group, $(X,\mu,T)$ is a probability $G$-space and $f:G\times X\to [0,\infty)$ is an integrable sub-cocycle, then there is some $M \geq 1$ (depending on $G$, $B_G$ and $f$) such that
\[\mu\{|f(g,x)| \geq M |g|_G\} \to 0 \quad \hbox{as} \quad |g|_G \to \infty.\]
\end{prop}

That is, as one considers increasingly large distances in $G$, the function $f$ is vanishingly unlikely to blow up those distances by any factor greater than $M$.  Note the convention that we always choose $M \geq 1$, even if one could actually use a smaller $M$ for some $f$.

The proof of Proposition~\ref{prop:as} rests on two basic geometric facts about nilpotent groups.

\begin{prop}[Approximation by straight-line segments]\label{prop:sla}
Suppose $G$ is a f.-g. nilpotent group with a finite symmetric generating set $B_G$.  Then there is some $K > 0$, depending on $G$ and $B_G$, with the following property. Whenever $g \in G$ with $|g|_G = n$, there is a $B_G$-word of the form
\[s_1^{a_1}\ldots s_k^{a_k},\quad a_1,a_2,\ldots,a_k \geq 0,\]
which evaluates to $g$ (where $s_1$, $s_2$, \ldots, $s_k$ are members of $B_G$ but may not be distinct) and such that
\[ k\leq K\quad \hbox{and}\quad a_1 + a_2 + \ldots + a_k \leq Kn.\]
\end{prop}

Intuitively, this asserts that `any point in $(G,d_G)$ may be reached by a sequence of at most $K$ straight-line segments of length not much greater than the distance to that point'. I have not been able to find a reference for Proposition~\ref{prop:sla}, but it is a fairly routine exercise in nilpotency, so its proof is deferred to Appendix~\ref{app:nilpcalc}.

\begin{rmk}
Conversely, any group $G$ having this property for some generating set $B_G$ must have polynomial growth with exponent at most $K$, and hence be virtually nilpotent, by Gromov's Theorem.  This follows by counting how many possible products there are of the form $s_1^{a_1}\ldots s_k^{a_k}$.  \fin
\end{rmk}

The second estimate we will need is the following.

\begin{prop}[Commutators grow sub-linearly]\label{prop:cc}
Let $G$ and $|\cdot|_G$ be as before.  Then for any $g,h \in G$ one has
\[|g^nhg^{-n}|_G = \rm{o}(n)\quad\hbox{as}\ n\to\infty\]
(although, of course, not uniformly in the choice of $g$ and $h$). \qed
\end{prop}

Proposition~\ref{prop:cc} is a special instance of de Cornulier's Proposition 3.1, part (iii$'$), and Corollary A.2 in~\cite{deC11}.  This is because, in his notation, the constant sequence $(h)$ is an element of $\rm{Sublin}(G)$ (whose definition can be found in that paper).  (Note that his Corollary A.2 seems to be mis-labelled as `A.7' in some versions.)  

Assuming the above group-theoretic facts, the next step towards Proposition~\ref{prop:as} is the following consequence of the Ergodic Theorem:

\begin{lem}\label{lem:cgce}
If $T:G\actson (X,\mu)$ is ergodic then for any $g \in G$ the functions
\[\frac{1}{n}f(g^n,\,\cdot\,)\]
converge $\mu$-a.e. as $n\to\infty$ to a function which is $\mu$-a.s. constant with value at most $\|f(g,\,\cdot\,)\|_1$.
\end{lem}

\begin{proof}
Since one always has
\[f(g^{n+m},x) \leq f(g^n,T^{g^m} x) + f(g^m,x),\]
the a.s. convergence follows from the Subadditive Ergodic Theorem.  This also gives that the limit is invariant under the subgroup $g^\bbZ \leq G$, but to prove a.s. constancy we need invariance under the action of the whole of $G$.  To this end, observe that if $h \in G$ then
\begin{eqnarray*}
f(g^n,T^hx) &=& f\big((g^nhg^{-n})g^nh^{-1},T^hx\big) \\
&\leq& f(g^nhg^{-n},T^{g^n}x) + f(g^n,x) + f(h^{-1},T^hx).
\end{eqnarray*}
The last right-hand term here is bounded in $L^1$, and the first term has $L^1$-norm which is $\rm{O}(|g^nhg^{-n}|_G) = \rm{o}(n)$, by Proposition~\ref{prop:cc}.  Therefore, dividing by $n$ and letting $n\to\infty$, we obtain
\[\lim_{n\to\infty} \frac{1}{n}f(g^n,T^hx) \leq \lim_{n\to\infty} \frac{1}{n}f(g^n,x).\]
Since we may clearly reverse this argument, the limit is actually $G$-invariant and hence a.s. constant.

The bound by $\|f(g,\cdot)\|_1$ is obvious from the triangle inequality.
\end{proof}

\begin{proof}[Proof of Proposition~\ref{prop:as}]
Let $K\geq 1$ be the constant appearing in Proposition~\ref{prop:sla} for $(G,d_G)$, and let
\[M := 4K^2\max_{s\in B_G}\|f(s,\,\cdot\,)\|_1.\]

Let $\eps > 0$, and first choose $n_0 \geq 1$ so large that
\[\mu\{f(s^n,x) \geq 2n\|f(s,\,\cdot\,)\|_1\} \leq \eps/2K\quad\quad\forall n\geq n_0,\ s\in B_G;\]
this is possible by Lemma~\ref{lem:cgce}.

Now suppose that $g \in G$, let $n := |g|_G$ and invoke Proposition~\ref{prop:sla} to obtain a $B_G$-word
\[g = s_1^{a_1}s_2^{a_2}\cdots s_k^{a_k}\]
with $k\leq K$ and length at most $Kn$ that evaluates to $g$.  We will show that
\[\mu\{f(g,x) \geq Mn\} < \eps\]
provided only that $n$ is sufficiently large.

Using the $B_G$-word above, we have
\[\frac{1}{n}f(g,x) \leq \sum_{j=1}^k \frac{1}{n}f(s_j^{a_j},T^{s_{j+1}^{a_{j+1}}\cdots s_k^{a_k}}x) = \sum_{j=1}^k \frac{a_j}{n}\frac{1}{a_j}f(s_j^{a_j},T^{s_{j+1}^{a_{j+1}}\cdots s_k^{a_k}}x).\]
Partition the set $\{1,2,\ldots,k\}$ as $I \cup I^{\rm{c}}$ with
\[I:= \{j\in \{1,2,\ldots,k\}\,|\ a_j \geq n_0\},\]
and consider the right-hand sum above decomposed as
\[\sum_{j \in I} \frac{a_j}{n}\frac{1}{a_j}f(s_j^{a_j},T^{s_{j+1}^{a_{j+1}}\cdots s_k^{a_k}}x) + \sum_{j \in I^{\rm{c}}} \frac{a_j}{n}\frac{1}{a_j}f(s_j^{a_j},T^{s_{j+1}^{a_{j+1}}\cdots s_k^{a_k}}x).\]
We will show that each of these two sub-sums can take abnormally large values only with very small probability.

\vspace{7pt}

\emph{First term}\quad Since $j \in I$ we have $a_j \geq n_0$, and hence
\[\mu\{f(s_j^{a_j},y) \geq 2a_j\|f(s_j,\,\cdot,)\|_1\} < \eps/2K.\]
From this it follows that
\begin{eqnarray*}
&& \mu\Big\{\sum_{j \in I} \frac{a_j}{n}\frac{1}{a_j}f(s_j^{a_j},T^{s_{j+1}^{a_{j+1}}\cdots s_k^{a_k}}x) \geq M/2\Big\}\\
&& \leq \mu\Big(\bigcup_{j \in I}\Big\{\frac{1}{a_j}f(s_j^{a_j},T^{s_{j+1}^{a_{j+1}}\cdots s_k^{a_k}}x) \geq (n/a_j)M/2|I|\Big\}\Big)\\
&& \leq \sum_{j\in I} \mu\Big\{\frac{1}{a_j}f(s_j^{a_j},y) \geq 2\max_{s\in B_G}\|f(s,\,\cdot\,)\|_1\Big\}\\
&&\leq K(\eps/2K) = \eps/2,
\end{eqnarray*}
where the deduction of the third line uses that $a_j \leq Kn$ and hence
\[(n/a_j)M/2|I| \geq (n/a_j)M/2K \geq M/2K^2 = 2\max_{s\in B_G}\|f(s,\cdot)\|_1.\]

\vspace{7pt}

\emph{Second term}\quad On the other hand, if $j \in I^{\rm{c}}$, then $a_j \leq n_0$, and hence
\[\sum_{j \in I^{\rm{c}}} \frac{a_j}{n}\frac{1}{a_j}f(s_j^{a_j},T^{s_{j+1}^{a_{j+1}}\cdots s_k^{a_k}}x) \leq \frac{n_0}{n}\sum_{j \in I^{\rm{c}}} \frac{1}{a_j}f(s_j^{a_j},T^{s_{j+1}^{a_{j+1}}\cdots s_k^{a_k}}x).\]
Integrating and using the triangle inequality, this function has $L^1$-norm at most
\[\frac{n_0}{n}\cdot |I^{\rm{c}}|\cdot \frac{1}{a_j}\cdot a_j\|f(s_j,\,\cdot\,)\|_1 \leq \frac{Mn_0}{n},\]
and so Markov's Inequality gives
\[\mu\Big\{\sum_{j \in I^{\rm{c}}} \frac{a_j}{n}\frac{1}{a_j}f(s_j^{a_j},T^{s_{j+1}^{a_{j+1}}\cdots s_k^{a_k}}x) \geq M/2\Big\} \leq \frac{2n_0}{n}\]
Provided we chose $n$ sufficiently large, this is at most $\eps/2$, and so combining this with our bound for the first term gives that
\[\mu\{f(g,x) \geq Mn\} < \eps/2 + \eps/2 = \eps,\]
as required.  This completes the proof.
\end{proof}

It might be interesting to study the generalization of Proposition~\ref{prop:as} to other groups.

\begin{ques}
For which groups and word metrics $(G,d_G)$ is it the case that for any probability $G$-space $(X,\mu,T)$ and any integrable sub-cocycle $f:G\times X\to [0,\infty)$ the functions $f(g,\cdot)$ must become asymptotically stable in distribution in the sense given by Proposition~\ref{prop:as} for some $M$? \fin
\end{ques}

We will make use of Proposition~\ref{prop:as} mostly through the following.

\begin{cor}\label{cor:interpt}
Let $G$ and $H$ be f.-g. nilpotent groups with word metrics $d_G$ and $d_H$, let $(X,\mu,T)$ be a probability $G$-space, and let $\a:G\times X\to H$ be an integrable cocycle over $T$.  Then for any $\eps > 0$ and $N \in \bbN$ there is some $C = C(\eps,N)$ such that, whenever $F\subset G$ has
$|F| = N$ and is $C$-separated for the metric $d_G$, one has \[\mu\big\{d_H(\a_x(g),\a_x(g')) \leq 2Md_G(g,g')\ \forall g,g'\in F\big\} > 1-\eps,\] where $M$ is the constant of
Proposition~\ref{prop:as} for $f:= |\a|_H$. \end{cor}

\begin{proof}
This follows by writing
\begin{eqnarray*}
&&\mu\big\{d_H(\a_x(g),\a_x(g')) > 2Md_G(g,g')\ \hbox{for some}\ g,g'\in F\big\}\\
&&\leq \sum_{g,g' \in F}\mu\big\{d_H(\a_x(g),\a_x(g')) > 2Md_G(g,g')\big\}\\
&&= \sum_{g,g' \in F} \mu\big\{d_H\big(\a_x(g),\a_{T^gx}(g'g^{-1})\a_x(g)\big) > 2Md_G(g,g')\big\}\\
&&= \sum_{g,g' \in F} \mu\big\{|\a_{T^gx}(g'g^{-1})|_H >
2M|g'g^{-1}|_G\big\},
\end{eqnarray*}
and now applying Proposition~\ref{prop:as} with error tolerance $\eps/N^2$.
\end{proof}

At one point, it will be more convenient to use Proposition~\ref{prop:as} through the following corollary.

\begin{cor}\label{cor:as}
In the setting of Proposition~\ref{prop:as}, and with $M$ the constant given there, it holds that for any $\eps > 0$ there is some $R = R(\eps)$ such that
\[\mu\{|f(g,x)| \geq M |g|_G + R\} < \eps \quad \forall g \in G.\]
(That is, we remove the assumption that $|g|_G$ be large by allowing an additive error.)
\end{cor}

\begin{proof}
Proposition~\ref{prop:as} gives $C > 0$ such that if $|g|_G \geq C$ then the result holds even without $R$.  The remaining cases follow by Markov's Inequality applied to the finite collection of integrable random variables $\{f(g,\cdot)\,|\ g \in B_G(e_G,C)\}$.
\end{proof}

\section{Completion of the proof}

Now consider again two f.-g. nilpotent groups $G,H$ and their asymptotic cones $\Con G$ and $\Con H$.  It remains to prove Theorem~\ref{thm:main2}: we must find some $L > 0$ such that for each $R > 0$ there are a set $E$ and map $\phi$ with the properties asserted there.

This map $\phi$ will be obtained from the restriction of the cocycle $\a_x$ to a suitable finite subset of $G$ for a `typical' point $x$.

As usual, we fix generating sets $B_G \subset G$ and $B_H \subset H$, which will become the $1$-balls in the resulting metrics $d_G$ and $d_H$.  The sequence of renormalized metric spaces $(G,n^{-1}d_G)$ converges in the local Gromov-Hausdorff sense to $(\Con G,d^G_\infty)$
as $n\to\infty$, and similarly for $(H,n^{-1}d_H)$.  This implies that for any finite subset $E \subset \Con G$ we can find a sequence of finite subsets $E_n \subset G$, $|E_n| = |E|$, and bijections $\phi_n:E_n\to E$ such that for any $c > 1$ one has
\[c^{-1}n^{-1}d_G(\phi_n(\ol{g}),\phi_n(\ol{g}')) \leq d^G_\infty(\ol{g},\ol{g}') \leq c n^{-1}d_G(\phi_n(\ol{g}),\phi_n(\ol{g}')) \quad \forall \ol{g},\ol{g}' \in E\]
for all sufficiently large $n$, and similarly for $H$ and $\Con H$.  Let us refer to such a sequence of maps $\phi_n$ as a sequence of \textbf{asymptotic copies} of $E$.  Since $G$ and $H$ are groups, by translating if necessary we may always assume that $E \ni \ol{e}_G$, $E_n \ni e_G$ for each $n$, and $\phi_n(e_G) = \ol{e}_G$; we will refer to such $E$ and $\phi_n$ as \textbf{pointed}.

For the proof, fix $R > 0$, and let $E$ be a pointed $(1/R)$-net in $B_G^\infty(R)$ (that is, an inclusion-maximal $(1/R)$-separated subset of this ball, which is therefore also $(1/R)$-dense in the ball).  Also, let $\phi_n: E \to E_n$ be a pointed sequence of asymptotic copies of $E$.

Theorem~\ref{thm:main2} will be a consequence of the following asymptotic behaviour of the cocycle $\a$.  Recall that a sequence of events $X_n$ in a probability space $(X,\mu)$ is said to occur \textbf{with high probability} (`\textbf{w.h.p.}') in $\mu$ if $\mu(X_n) \to 1$.

\begin{thm}\label{thm:main3}
Let $M$ be the maximum of the two constants obtained by applying Proposition~\ref{prop:as} to
$\a$ and to $\b$. Then as $n\to\infty$, all of the following hold w.h.p. in $\mu_X$:
\begin{itemize}
\item[i)] $\a_x|E_n$ is $(2M)$-Lipschitz;
\item[ii)] $\a_x|E_n$ is $(4M)$-co-Lipschitz;
\item[iii)] $\a_x(E_n)$ is $(6Mn/R)$-dense in $B_H(nR/8M)$.
\end{itemize}
\end{thm}

\begin{proof}[Proof of Theorem~\ref{thm:main2} from Theorem~\ref{thm:main3}]
In addition to $E$ and $\phi_n$, let $F \subseteq B^\infty_H(4MR)$ be a pointed $(1/32MR)$-net and $\psi_n:F\to F_n\subseteq H$ a sequence of pointed asymptotic copies of it.

By properties (i) and (ii) above, as $n\to\infty$, it holds w.h.p. in $\mu_X$ that
\[\a_x(E_n) \subseteq B_G(3MRn) \quad \hbox{and} \quad \min_{g,g' \in E_n\,\rm{distinct}}d_H(\a_x(g),\a_x(g')) \geq \frac{1}{4MR}n.\]
For each $n$ and $x$, let
\[\eta_x:\a_x(E_n)\to F_n\]
be a map such that, for every $g \in E_n$, $\eta_x(\a_x(g))$ is an element of $F_n$ at minimal distance from $\a_x(g)$.  In view of the above properties of $\a_x(E_n)$, and by the density of $F$, it holds w.h.p. that $\eta_x$ is injective, and that if $g,g' \in E_n$ are distinct then
\begin{eqnarray*}
\frac{1}{2}d_H(\a_x(g),\a_x(g')) &\leq& d_H(\a_x(g),\a_x(g')) - \frac{2}{32MR}n\\ &\leq& d_H\big(\eta_x(\a_x(g)),\eta_x(\a_x(g'))\big) \\
&\leq& d_H(\a_x(g),\a_x(g')) + \frac{2}{32MR}n\\ &\leq& 2d_H(\a_x(g),\a_x(g')).
\end{eqnarray*}
Having seen this, it follows that w.h.p. in $\mu_X$ the composition
\[\phi\ :\ E \stackrel{\phi_n}{\to} E_n \stackrel{\a_x}{\to} \a_x(E_n) \stackrel{\eta_x}{\to} F_n \stackrel{\psi_n^{-1}}{\to} F\]
is both $(8M)$-Lipschitz and $(8M)$-co-Lipschitz once $n$ is sufficiently large.  Also,
\[\phi(\ol{e}_G) = \psi_n^{-1}(\eta_x(\a_x(e_G))) = \psi_n^{-1}(\eta_x(e_H)) = \psi_n^{-1}(e_H) = \ol{e}_H,\]
because $e_H$ must be the unique point of $F_n$ closest to itself.

Therefore, the proof will be completed upon showing that $\phi(E)$ is $(32M/R)$-dense in $B^\infty_H(R/32M)$.  This follows by property (iii), and the fact that $\eta_x$ does not move any point of $\a_x(E_n)$ by a distance greater than $(1/16MR)n$, which implies that $\eta_x(\a_x(E_n))$ is still $(16Mn/R)$-dense in $B_H(nR/8M)$.
\end{proof}

Property (i) of Theorem~\ref{thm:main3} follows directly from Corollary~\ref{cor:interpt}.  Properties (ii) and (iii) will need a little more work.  For these we will also need to use related estimates for the cocycle $\b$ going in the other direction.

In case our original measure coupling gives $X = Y = X\cap Y$, so that $\b_x = \a_x^{-1}$ for all $x$, property (ii) looks very like property (i) with $\a$ replaced by $\b$.  However, even in this special case, there is an extra subtlety here.  Property (ii) is asserting that
\begin{quote}
$\b_x|\a_x(E_n)$ is $(4M)$-Lipschitz.
\end{quote}
This differs from property (i) in that the relevant domain, $\a_x(E_n)$, now also depends on $x$.  This will force us to use a more careful argument than for Corollary~\ref{cor:interpt}, because we must rule out the possibility that, as $x$ varies, the set-valued function $x\mapsto \a_x(E_n)$ always happens to choose a set on which $\b_x$ behaves irregularly.  To rule this out, we will choose a new fixed set $F_n \subset H$ which is $(\delta n)$-dense for some $\delta \ll 1/R$, and show that w.h.p. the restriction $\b_x|\a_x(E_n)$ stays very close to the restriction of $\b_x$ to a set of points in $F_n$ that lie nearby the points in $\a_x(E_n)$.  On the other hand, the analog of (i) will give that $\b_x$ is $(2M)$-Lipschitz on the whole of $F_n$, and from this we can then gain control of the Lipschitz constant of its restriction to $\a_x(E_n)$, notwithstanding that dependence on $x$.  At the end of this section we will present an example showing that cocycles such as $\a_x$ can have occasional `defects' where their behaviour is very far from Lipschitz, which suggests that this extra care is really needed.

A similar comparison with $\b_x|F_n:F_n\to G$ will also underly the proof of property (iii).

The first step is the following.

\begin{lem}
Let $x\mapsto D_x$, $y\mapsto E_y$, $\mu_X$ and $\mu_Y$ be as in Subsection~\ref{subs:basics}.  Then
\[\frac{|D_x\cap B_G(r)|}{|B_G(r)|} \to \mu_X(X\cap Y) \quad \hbox{as}\ r\to\infty\]
in $L^1(\mu_X)$ (regarding the left-hand side as a function of $x$), and similarly
\[\frac{|E_y\cap B_H(r)|}{|B_H(r)|} \to \mu_Y(X\cap Y) \quad \hbox{as}\ r\to\infty\]
in $L^1(\mu_Y)$.
\end{lem}

\begin{proof}
For f.-g. nilpotent groups such as $G$ and $H$, another result from~\cite{Pan83} is that the polynomial growth rate of radius-$r$ balls is very exact, in the sense that $|B_G(r)|/r^{d_G}$ tends to a fixed positive limit as $r\to\infty$ for some integer $d_G > 0$, and similarly for $|B_H(r)|$.  This implies that the balls $B_G(r)$ (resp. $B_H(r)$) form a F\o lner sequence in $G$ (resp. $H$) as $r\to\infty$.  The result now follows from the Norm Ergodic Theorem for the $G$- (resp. $H$-) action and the fact that $T$ (resp. $S$) is ergodic.
\end{proof}

The next lemma asserts that once the radius $R$ is sufficiently large, for most $x$ the ball-image $\a_x(B_G(g,R)) \subset H$ must be mostly contained inside the slightly larger ball $B_H(\a_x(g),2MR)$.

\begin{lem}[Controlling images of balls]\label{lem:ball-control} Let $M$ be the maximum of the two
constants obtained by applying Proposition~\ref{prop:as} to
$\a$ and to $\b$. Then for any $\eps > 0$ and $g \in G$, the following holds w.h.p. in $\mu_X$ as $R\to\infty$:
\[\big|B_G(g,R) \cap \a_x^{-1}(B_H(\a_x(g),2MR))\big| \geq
(1-\eps)|B_G(g,R)|.\]
The same holds with the r\^oles of $(G,g,\a_x)$ and $(H,h,\b_x)$ reversed.
\end{lem}

\begin{proof}
This will follow from Markov's
Inequality if we prove instead that
\[\sum_{g' \in B_G(g,R)}\mu_X\{d_H(\a_x(g'),\a_x(g)) \leq 2MR\}
\geq \sqrt{1 - \eps}|B_G(g,R)|.\]
However, by the invariance of
$\mu_X$ and the cocycle identity for $\a$, the left-hand summands here are equal to
\[\mu_X\{|\a_{T^gx}(g'g^{-1})|_H \leq 2MR\} = \mu_X\{|\a_x(g'g^{-1})|_H \leq 2MR\}\]
for $g'\in B_G(g,R)$, and to each of these summands we may apply Proposition~\ref{prop:as}.
\end{proof}

We will now combine the estimates of the previous two lemmas into the following conclusion.  It will be the key to controlling both the typical co-Lipschitz constant of $\a_x|E_n$ and the density of its image.

\begin{prop}\label{prop:ball-compar}
For every $\eps > 0$ there exists $R_0$ such that for all $g \in G$, $h \in H$ and $R \geq R_0$ one has
\[\mu_X\{d_H(\a_x(g),h) \leq R,\ d_G(g,\b_x(h)) > 5MR\} < \eps.\]
The same holds with the r\^oles of $(G,g,\a_x)$ and $(H,h,\b_x)$ reversed.
\end{prop}

\begin{proof}
The key to this is a volume comparison of certain balls around $g$ and $h$ and their $\a_x$- or $\b_x$-images.  It is easiest to explain the idea in the special case $X = Y = X\cap Y$, so that $D_\bullet \equiv G$ and $E_\bullet \equiv H$.  In that case, if $R$ is large enough, then $\a_x$ typically maps most of the $(R/2M)$-ball around $g$ into the $R$-ball around $\a_x(g)$, by Lemma~\ref{lem:ball-control}.  If $d_H(\a_x(g),h) \leq R$, then that $\a_x$-ball-image will occupy a significant fraction of the $(2R)$-ball around $h$, because $d_G$ and $d_H$ are doubling and have the same growth rate (Lemma~\ref{lem:growth-est}).  Now another appeal to Lemma~\ref{lem:ball-control}, this time for $\b_x = \a_x^{-1}$, shows that the $\b_x$-image of the $(2R)$-ball around $h$ typically lands almost entirely in the $(4MR)$-ball around $\b_x(h)$.  Combining these facts, it follows that some positive fraction of $\a_x(B_G(g,R/2M))$ usually also lands in that last ball.  This implies, in particular, that $B_G(g,R/2M)$ and $B_G(\b_x(h),4MR)$ must intersect, and this then implies that $d_G(g,\b_x(h)) \leq 4MR + R/2M \leq 5MR$.

In general we argue as follows. By Lemma~\ref{lem:ball-control}, for any $\eps > 0$, all of the following events occur w.h.p. in $\mu_X$ as $R \to \infty$, uniformly in the choice of $g$ and $h$:
\[\big\{|B_G(g,R/2M)\cap \a_x^{-1}(B_H(\a_x(g),R))| \geq (1-\eps)|B_G(g,R/2M)|\big\},\]
\[\big\{|B_H(h,2R)\cap \b_x^{-1}(B_G(\b_x(h),4MR))| \geq (1 - \eps)|B_H(h,2R)|\big\},\]
\[\big\{|D_x\cap B_G(g,R/2M)| \geq (\mu_X(X\cap Y) - \eps)|B_G(g,R/2M)|\big\},\]
and
\[\big\{|E_x\cap B_H(h,R)| \geq (\mu_Y(X\cap Y) - \eps)|B_H(h,R)|\big\}.\]

We will show that on the intersection of these events, either
\[d_H(\a_x(g),h) > R\]
or
\[d_H(\a_x(g),h) \leq R \quad \hbox{and} \quad d_G(g,\b_x(h)) \leq 5MR.\]

Thus, assume that $x$ lies in this intersection and that $d_G(\a_x(g),h) \leq R$.  This implies that $B_H(h,2R) \supseteq B_H(\a_x(g),R)$, and hence
\begin{eqnarray}\label{eq:volest1}
&&|B_H(h,2R)\cap \a_x(D_x\cap B_G(g,R/2M))| \nonumber\\
&&= |\a_x^{-1}(B_H(h,2R)) \cap D_x\cap B_G(g,R/2M)| \nonumber\\
&&\geq |D_x\cap B_G(g,R/2M)| - \eps |B_G(g,R/2M)| \nonumber\\
&&\geq (\mu_X(X\cap Y) - 2\eps)|B_G(g,R/2M)|,
\end{eqnarray}
using the fact that $\a_x|D_x$ is an injection for the first equality.  Using that $\b_x|E_x$ is injective, for such $x$ one similarly obtains
\begin{multline}\label{eq:volest2}
|E_x\cap B_H(h,2R)\cap \b_x^{-1}(B_G(\b_x(h),4MR))|\\ \geq (\mu_Y(X\cap Y) - 2\eps)|B_H(h,2R)|,
\end{multline}
and finally
\begin{eqnarray}\label{eq:volest3}
|E_x\cap B_H(h,2R)| \leq (\mu_Y(X\cap Y) + \eps)|B_H(h,2R)|.
\end{eqnarray}

Now, by Lemma~\ref{lem:growth-est}, there is some $D > 0$ such that
\[|B_G(g,R/2M)| \geq D|B_H(h,2R)| \quad \forall R > 0.\]
Therefore, if $\eps$ is small enough then the sum of the right-hand sides of~(\ref{eq:volest1}) and~(\ref{eq:volest2}) is strictly greater than the right-hand side of~(\ref{eq:volest3}), implying that
\[\a_x(D_x \cap B_G(g,R/2M))\cap E_x\cap \b_x^{-1}(B_G(\b_x(h),4MR)) \neq \emptyset.\]
Letting $k =\a_x(\b_x(k))$ be an element of this set, the triangle inequality gives
\begin{eqnarray*}
d_G(g,\b_x(h)) &\leq& d_G(g,\b_x(k)) + d_G(\b_x(k),\b_x(h))\\
&\leq& R/2M + 4MR \leq 5MR,
\end{eqnarray*}
as required.
\end{proof}

\begin{proof}[Proof of Theorem~\ref{thm:main3}]
As remarked previously, property (i) follows from Corollary~\ref{cor:interpt}, so it remains to prove (ii) and (iii).

Recall that $E\subset G$ is a fixed pointed $(1/R)$-dense subset of the ball $B^\infty_G(R)$, and that $\phi_n:E\to E_n$ are pointed asymptotic copies of it.  Now choose in addition a pointed $(1/100M^2R)$-dense subset $F$ of $B^\infty_H(100M^2R)$, and a sequence $\psi_n:F\to F_n$ of pointed asymptotic copies of it.

\vspace{7pt}

\emph{Proof of (ii).}\quad Since $|E_n| = |E|$ is fixed, it will suffice to prove that for
any $\eps > 0$ there is some $n_0 > 0$ such that
\[\mu_X\{d_G(g,g') \leq 4Md_H(\a_x(g),\a_x(g'))\} > 1 - \eps\]
whenever $n \geq n_0$ and $g,g' \in E_n$ are distinct.  Letting $k := g'g^{-1}$, and using the cocycle relation, the right-invariance of the metrics, and the $T$-invariance of $\mu_X$, this measure is equal to
\begin{multline*}
\mu_X\{|k|_G \leq 4Md_H(\a_x(g),\a_{T^gx}(k)\a_x(g))\}\\ = \mu_X\{|k|_G \leq 4Md_H(\a_{T^gx}(k),e)\} = \mu_X\{|k|_G \leq 4M|\a_x(k)|_H\}.
\end{multline*}
The length $|k|_G$ lies between $n/3R$ and $3Rn$ for all sufficiently large $n$, so the result will follow if we show that
\[n/3R \leq |k|_G \leq 3Rn \quad \Longrightarrow \quad \mu_X\{|k|_G \leq 4M|\a_x(k)|_H\} > 1 - \eps\]
for all sufficiently large $n$.

To do this, observe that for $n$ sufficiently large one can find $h \in F_n$ such that $d_H(\a_x(k),h) \leq n/99M^2R$.  Since $|F_n| = |F|$ is fixed, we may combine this fact with Corollary~\ref{cor:as} and Proposition~\ref{prop:ball-compar} to deduce that the event
\begin{multline*}
\big\{\exists h \in F_n\ \hbox{such that}\ d_H(\a_x(k),h) \leq n/99M^2R,\\ d_H(k,\b_x(h)) \leq n/19MR,\ \hbox{and}\ |\b_x(h)|_G\leq 2M|h|_H + n/1000MR\big\}
\end{multline*}
occurs w.h.p. in $\mu_X$ as $n\to\infty$.  On this event, choosing a suitable $h \in F_n$, two applications of the triangle inequality give
\begin{eqnarray*}
|k|_G &\leq& |\b_x(h)|_G + d_G(k,\b_x(h)) \leq 2M|h|_H + n/1000MR + n/19MR\\ &\leq& 2M|\a_x(k)|_H + 2Md_H(\a_x(k),h) + n/18MR\\ &\leq& 2M|\a_x(k)|_H + n/10MR,
\end{eqnarray*}
and hence $2M|\a_x(k)|_H \geq |k|_G - n/10MR \geq |k|_G/2$, as required.

\vspace{7pt}

\emph{Proof of (iii).}\quad Now fix $h \in B_H(nR/2M)$, and consider its image $\b_x(h) \in G$.  Both of the following hold w.h.p. as $n\to \infty$:
\begin{itemize}
\item $|\b_x(h)|_G \leq nR$, and hence $\exists g \in E_n$ such that $d_G(g,\b_x(h)) \leq n/R$;
\item for all $g' \in E_n$, one has
\[\hbox{either} \quad d_G(g',\b_x(h)) > n/R \quad \hbox{or} \quad d_H(\a_x(g'),h) \leq 5Mn/R.\]
\end{itemize}
On the intersection of these events, it follows that
\begin{eqnarray}\label{eq:h-in-nhood}
h \in B_H(\a_x(E_n),5Mn/R)
\end{eqnarray}
Letting $\psi:F\to F_n$ be a sequence of asymptotic models for a $(Mn/2R)$-dense subset of $B^\infty_H(R/2M)$, and observing that $|F_n| = |F|$ is fixed, it follows that, w.h.p. in $\mu_X$ as $n\to\infty$, the containment~(\ref{eq:h-in-nhood}) holds simultaneously for all $h \in F_n$.  On this event, the image $\a_x(E_n)$ is $(5Mn/R + Mn/2R + \rm{o}(1)n)$-dense, hence $(6Mn/R)$-dense, in $B_H(nR/2M)$, as required.
\end{proof}

This completes the proof of our main theorems.  Before leaving this section, it is worth including an example of an IME in which the cocycle $\a_x$ exhibits occasional bad behaviour at arbitrarily large scales for a.e. $x$.  This justified the care we have taken over the proofs of properties (ii) and (iii) above.

\begin{ex}
We will construct an integrable orbit equivalence between two $\bbZ^2$-actions (so $D_\bullet = E_\bullet = \bbZ^2$).  As recalled in Subsection~\ref{subs:basics}, we can do this by constructed instead a suitable probability measure $\mu$ on
\[X := [\bbZ^2,\bbZ^2] := \{\a:\bbZ^2\to \bbZ^2\,|\ \a(0) = 0\}.\]
This measure $\mu$ should be supported on the subset of bijections, and be invariant under the action $T$ of $\bbZ^2$ defined by $T^v\a(w) = \a(w+v) - \a(v)$, which we call the \textbf{adjusted translation action}.  For an integrable orbit equivalence, it must also satisfy
\begin{eqnarray}\label{eq:integrability}
\max_{i=1,2}\int_X |\a(e_i)| + |\a^{-1}(e_i)|\,\mu(\d \a) < \infty,
\end{eqnarray}
where $e_1,e_2$ is the standard basis in $\bbZ^2$, and $|\cdot|$ is the $\ell_1$-distance on $\bbZ^2$.

This random element of $X$ will be constructed as a limit in the following way.  For each $m$, let $\rho_m$ be the law of a random subset $S_m\subseteq \bbZ^2$ in which each point is included independently with probability $4^{-m}$.  Thus, each $\rho_m$ is a translation-invariant probability on $\{0,1\}^{\bbZ^2}$.  Now, for each of these subsets $S_m$, let $\k_m:\bbZ^2\to \bbZ^2$ be the bijection defined as follows:
\begin{itemize}
 \item if $v \in S_m$ and $v + 2^me_1 \not\in S_m$, then $\k_m$ swaps $v$ and $v + 2^me_1$;
\item $\k_m$ fixes all other points.
\end{itemize}
Each $\k_m$ is a random permutation of $\bbZ^2$ with translation-invariant law.  Letting $\a_m(v) := \k_m(v) - \k_m(0)$, this defines a random element of $[\bbZ^2,\bbZ^2]$ whose law is invariant under the adjusted translation action.

Finally, letting $(\a_1,\a_2,\ldots)$ be drawn at random from the product measure $\rho:= \rho_1\otimes \rho_2\otimes \cdots$, an easy estimate shows that for any fixed $v \in \bbZ^2$ the sequence
\[\a_M\circ \cdots \circ \a_3\circ \a_2\circ \a_1(v), \quad M=1,2,\ldots\]
is eventually constant with probability $1$ in the choice of $(\a_m)_m$.  Calling its eventual value $\a(v)$, this defines a random map $\bbZ^2 \to \bbZ^2$ which is a.s. bijection, sends $0$ to $0$, and has law that is invariant under the adjusted translation action.  Also, it satisfies
\[\rho\big\{(\a_{m'})_{m'}\,\big|\ |\a(e_i)| > 2^m\big\} < \sum_{m'\geq m}4^{-m'} \leq 4^{-m+1}\]
for $i=1,2$, and similarly for $\a^{-1}$, from which~(\ref{eq:integrability}) follows.

Finally, however, observe that for each $m$, in the box $[-2^{m+1},2^{m+1}]^2$, which contains roughly $4^{m+2}$ points, one has a positive probability that $\a_m$ will move at least one point by distance $2^m$.  Using the independence of $\a_1$, $\a_2$, \ldots under $\rho$, a simple Borel-Cantelli argument now implies that with $\rho$-probability $1$, $\a$ has the following property:
\begin{quote}
 There is an infinite sequence of scales $m_1 < m_2 < \ldots$ and, for every $i$, a pair of points $u,v \in [-2^{m_i+1},2^{m_i+1}]$ such that $|u-v| = 1$ but $|\a(u) - \a(v)| \geq 2^{m_i}$.
\end{quote}
Thus, at every length scale, there can be a few pairs of neighbouring points at which $\a$ is as `far from Lipschitz' as it could be.  The point of Proposition~\ref{prop:ball-compar} was to show that these bad points are so rare that we can simply work around them in Theorem~\ref{thm:main3}.  It is worth contrasting this with the arguments of~\cite{deC11}, which also construct bi-Lipschitz maps between cones from non-quasi-isometries between groups, but require a more uniform control on the bad behaviour of those maps of the groups. \fin
\end{ex}

\section{Remaining issues}

Most obviously, it would be interesting to know whether the results of this paper extend beyond the class of virtually nilpotent groups (I am confident that the methods do not).

\begin{ques}
For which pairs of amenable groups does an IME imply bi-Lipschitz-equivalent asymptotic cones (for some non-principal ultrafilters)?
\end{ques}

Among nilpotent groups, Theorem~\ref{thm:main} suggests another interesting line of enquiry.   For simplicity, consider a case in which $G$ and $H$ are both quasi-isometric to their asymptotic cones, say via maps $\phi:\Con G\to G$ and $\psi:H \to \Con H$.  Recall~(\cite{Pan83}) that the asymptotic cones are graded connected nilpotent Lie groups equipped with dilations $\delta_{\Con G}$ and $\delta_{\Con H}$.  Given an integrable measure equivalence implemented by the cocycles $\a$ and $\b$ as before, for each $x \in X$ and $n \geq 1$ one can consider the map
\[\kappa_{x,n}:\Con G\to \Con H:\bar{g} \mapsto \delta_{\Con H}^{1/n}(\psi(\a_x(\phi(\delta_{\Con G}^n(\bar{g}))))).\]

\begin{ques}
Is it true that for $\mu$-a.e. $x \in X$, $\kappa_{x,n}$ converges (say, in probability on bounded subsets of $\Con G$) to a bi-Lipschitz isomorphism of groups $\Con G\to \Con H$?
\end{ques}

If true, this would amount to a kind of `nilpotent-valued' version of the Pointwise Ergodic Theorem.  It has the flavour of a large-scale analog for cocycles of the problem of proving an analog of Rademacher's Theorem for Lipschitz maps between Carnot-Carath\'eodory metrics.  Such a differentiation theorem has been studied by Pansu in~\cite{Pan89} and Margulis and Mostow in~\cite{MarMos95}.

\appendix

\section{Approximation with straight-line segments}\label{app:nilpcalc}

The proof of Proposition~\ref{prop:sla} requires some preparations.  Let $G^1 := G$ and $G^{i+1} := [G,G^i]$ denote the descending central series of $G$, so that $G^{c+1} = \{e_G\}$ if $c$ is the nilpotency class of $G$.  The following requires only a routine calculation with commutators.

\begin{lem}\label{lem:gen-Gm}
If $B_G$ is a finite symmetric generating set for $G$, then for each $m \in \{2,3,\ldots,c\}$ a generating set for $G^m$ is given by set of the $m$-fold commutators
\[[s_1,[s_2,[\cdots[s_{m-1},s_m]\cdots]]],\quad s_1,s_2,\ldots,s_m \in B_G\]
and their inverses. \qed
\end{lem}

The next calculation is slightly less standard, so we include a proof for completeness.  A similar calculation in the setting of a nilpotent Lie algebra appears as Lemma 4.1 in Pittet~\cite{Pit97} and (see also Pansu~\cite{Pan83}).

\begin{lem}\label{lem:comms-manip}
If $G$ is a nilpotent group and $s_1,s_2,\ldots,s_m \in G$ then for all $n\geq 1$ one has
\[[s^n_1,[s^n_2,[\cdots[s^n_{m-1},s^n_m]\cdots]]] = \big([s_1,[s_2,[\cdots[s_{m-1},s_m]\cdots]]]\big)^{n^m}\cdot r_n\]
where $r_n \in G^{m+1}$.
\end{lem}

\begin{proof}
We fix $n\geq 1$ and prove this assertion by induction on $m$.  For each $m$, it suffices to treat the case when $G$ has nilpotency class at most $m$, since for general $G$ we may simply lift the desired result from the quotient $G/G^{m+1}$ (because $r \in G^{m+1}$ is allowed to be arbitrary).

The result is trivial when $m=1$, so assume it is known for some $m$ and consider $s_1,s_2,\ldots,s_{m+1}\in G$, where $G$ has class at most $m+1$.  The inductive hypothesis gives
\[[s^n_2,[s^n_3,[\cdots[s^n_m,s^n_{m+1}]\cdots]]] = \big([s_2,[s_3,[\cdots[s_m,s_{m+1}]\cdots]]]\big)^{n^m}\cdot r\]
for some $r \in G^{m+1}$, so $r$ is central in $G$.  Let
\[g := [s_2,[s_3,[\cdots[s_m,s_{m+1}]\cdots]]],\]
so this is in $G^m$, and now insert the above expression into the commutator with $s_1^n$ to obtain
\begin{multline*}
[s^n_1,[s^n_2,[\cdots[s^n_m,s^n_{m+1}]\cdots]]] = [s_1^n,g^{n^m}r]\\ = s_1^ng^{n^m}rs_1^{-n}g^{-n^m}r^{-1} = s_1^ng^{n^m}s_1^{-n}g^{-n^m},
\end{multline*}
where the last equality uses that $r$ is central.  Now each appearance of $s_1$ on the left end of this word may be moved through the sub-word $g^{n^m}$ to cancel an appearance of $s_1^{-1}$, creating $n^m$ copies of the commutator $[s_1,g]$.  Since that commutator is in $G^{m+1}$ and so is central, it may be placed at the far left end of the resulting word.  Repeating this manipulation $n$ times, we obtain
\begin{multline*}
s_1^ng^{n^m}s_1^{-n}g^{-n^m} = [s_1,g]^{n^m}s_1^{n-1}g^{n^m}s_1^{-(n-1)}g^{-n^m}\\ = [s_1,g]^{2\cdot n^m}s_1^{n-2}g^{n^m}s_1^{-(n-2)}g^{-n^m} = \ldots = [s_1,g]^{n^{m+1}}.
\end{multline*}
This is the desired expression, so the induction continues.
\end{proof}

In order to make use of these results, we also need the following important calculation relating the word metrics of a f.-g. nilpotent group and of one of its subgroups.  In fact, it is a special case of a rather more general results on the possible distortions of the word metrics on subgroups of nilpotent groups, obtained by Osin as Theorem 2.2 in~\cite{Osi01}; see also Pittet~\cite{Pit97} and Subsection 3.B$_2$ of Gromov~\cite{Gro93}. 

\begin{lem}\label{lem:subgp-dist}
If $G$ is a f.-g. nilpotent group of nilpotency class $m$, and $B$ and $B'$ are finite generating sets of $G$ and $G^m$ respectively, then there is some constant $C$ such that
\[|h|_{B'} \leq C|h|_B^m\quad \forall h \in G^m.\] \qed
\end{lem}

\begin{proof}[Proof of Proposition~\ref{prop:sla}]
This follows from an induction on the nilpotency class of $G$.  When $G$ is Abelian the result is trivial, so suppose that $G$ has class $m\geq 2$.  Let $B \subseteq G$ be any finite symmetric generating set, let $\bar{B}$ be its image under the quotient map $G\to G/G^m$, and let $K$ be the constant implied by our assumption of Proposition~\ref{prop:sla} for $(G/G^m,d_{\bar{B}})$.  Let $g \in G$, and let $\bar{g} = gG^m$.  Then $\bar{B}$ is finite, symmetric and generates $G/G^m$, and clearly $|\bar{g}|_{\bar{B}} \leq |g|_B$, so by the inductive hypothesis there are $s_1,\ldots,s_k \in B$ and $a_1$, $a_2$, \ldots, $a_k \geq 0$ such that $k \leq K$, $\sum_i a_i \leq K|g|$ and
\[\bar{g} = \bar{s}_1^{a_1}\bar{s}_2^{a_2}\cdots \bar{s}_k^{a_k}.\]

Lifting back to $G$, this becomes
\[g = s_1^{a_1}s_2^{a_2}\cdots s_k^{a_k}\cdot h\]
for some
\[h = s_k^{-a_k}s_{k-1}^{-a_{k-1}}\cdots s_1^{-a_1}g \in G^m.\]
It follows that $|h|_B \leq (K+1)|g|_B$, and by Lemma~\ref{lem:gen-Gm} it may be expressed as a word in the $m$-fold commutators
\[[u_1,[u_2,[\cdots[u_{m-1},u_m]\cdots]]],\quad u_1,u_2,\ldots,u_m \in B\]
and their inverses.  Let $B'$ be the set of these commutators and their inverses.

Lemma~\ref{lem:subgp-dist} promises some constant $C$ such that
\[|h|_{B'} \leq C|h|_B^m\quad\forall h \in G_m.\]
Let $K'$ be the constant promised by the statement our proposition for the Abelian group $(G^m,d_T)$. Since $|h|_B \leq (K+1)|g|_B$, it follows that $h$ may be expressed as
\[t_1^{b_1}t_2^{b_2}\cdots t_\ell^{b_\ell}\]
for some distinct $t_1$, $t_2$, \ldots, $t_\ell \in B'$ and $b_1,\ldots,b_\ell \geq 0$ such that $\sum_i b_i \leq CK'(K+1)|g|^m$.

Now recall that according to the Hilbert-Waring Theorem, there is some $L \geq 1$ such that any positive integer may be written as a sum of at most $L$ perfect $m^{\rm{th}}$ powers.  Applying this to each $b_i$, we may instead express $h$ as a word
\[v_1^{n^m_1}v_2^{n^m_2}\cdots v_\ell^{n^m_\ell},\]
where now the $v_i$ are (not necessarily distinct) elements of $B'$, each $n_i$ is at most $(CK'(K+1))^{1/m}|g|_B$, and $\ell \leq L|B'|$.

Finally, if
\[v = [u_1,[u_2,[\cdots[u_{m-1},u_m]\cdots]]] \in B'\]
then Lemma~\ref{lem:comms-manip} enables one to write $v^{n^m}$ as
\[[u^n_1,[u^n_2,[\cdots[u^n_{m-1},u^n_m]\cdots]]].\]
Inserting such multiple commutators into the place of each power $v_i^{n_i^m}$ appearing in the word for $h$ above, it follows that $h$ can be expressed as a product of powers of elements of $B$ in which each power is at most $(K'(K+1))^{1/m}|g|_B$, and the number of powers appearing in the product is at most $4^mL|B'|$.  This completes the proof.
\end{proof}

\section{$L^1$-measure equivalence and group growth}

\begin{center}
by Lewis Bowen\footnote{supported in part by NSF grant DMS-0968762 and NSF CAREER Award DMS-0954606}
\end{center}

\def\gr{{\textrm{gr}}}
\def\cc{{\curvearrowright}}

\subsection{Introduction}

\begin{dfn}[Weak equivalence]\label{defn:weak}
Let $f,g$ be two real-valued functions of the natural numbers. We write $f \lesssim g$ if there are positive constants $C_1,C_2$ such that $f(n) \le C_1 g(C_2n)$  for all sufficiently large $n$. We say $f$ and $g$ are {\em weakly equivalent}, denoted $f \approx g$, if $f \lesssim g$ and $g \lesssim f$. 
\end{dfn}

Let $G$ be a finitely generated group. Let $\gr_G(n)=|B_G(e,n)|$ be the number of elements in the ball of radius $n$ of $G$ (with respect to a fixed word metric). The function $\gr_G$ depends on the choice of generating set only up to weak equivalence. Its weak equivalence class is called the {\em degree of growth} of $G$. This notion was introduced by A. S. Schwarz (spelled also as Schvarts and \v Svarc) \cite{Sva55} and independently by Milnor \cite{Mil68, Mil68b}. For a recent survey on growth of groups, see \cite{Gri13}.


Our main result is:

\begin{thm}\label{thm:growth2}
Let $G$, $H$ be two finitely generated IME groups. Then $\gr_G \approx \gr_H$. 
\end{thm}

\begin{cor}
There is an uncountably family of non-IME countably infinite amenable groups.
\end{cor}

\begin{proof}
In \cite{Gri85} it is shown that there exists an uncountable family of degrees of growth of groups. These groups are amenable since all non-amenable groups have the same degree of growth, namely exponential growth.
\end{proof}

By contrast, it follows from work of Ornstein-Weiss \cite{OrnWei80} (extending well-known results of Dye \cite{Dye59,Dye63}) that all countably infinite amenable groups are measure-equivalent.

We obtain Theorem \ref{thm:growth2} as a corollary to a more general result relating growth and integrable-embeddings of groups. This notion is developed in the next two sections.

\subsection*{Acknowledgements}

This note owes its inspiration and motivation from ongoing discussions with Tim Austin, Uri Bader, Alex Furman and Roman Sauer. I am grateful for their encouragement and helpful discussions.

\subsection{Measurable embeddings}

\begin{dfn}[Cocycles and cohomology]
Let $G \cc^T (X,\mu)$ be a finite-measure-preserving (fmp) action. Recall that a measurable map $\alpha: G\times X \to H$ is a {\em cocycle} over $T$ if
$$\alpha(g_2g_1,x)=\alpha(g_2,T^{g_1}x)\alpha(g_1,x)$$
for every $g_2,g_1\in G$ and a.e. $x\in X$. We say that two such cocycles $\alpha,\alpha'$ are {\em cohomologous} if there exists a measurable map $\phi:X \to H$ such that 
$$\alpha'(g,x) = \phi(T^gx) \alpha(g,x) \phi(x)^{-1}.$$
\end{dfn}

\begin{dfn}\label{defn:embedding}
Let $G \cc (X,\mu)$ be an fmp action and $\alpha:G \times X \to H$ a measurable cocycle. We say $\alpha$ is a {\em measurable embedding} if there is a measurable cocycle $\alpha':G\times X \to H$ cohomologous to $\alpha$ and a constant $C>0$ such that for every $h\in H$ and a.e. $x\in X$,
$$|\{g\in G:~ \alpha'(g,x)=h\}| \le C.$$
\end{dfn}

Although we are primarily interested in the $L^1$-version of this definition (given in the next section) here we justify this definition by showing that any cocycle associated to an ME-coupling is a measurable embedding.

\begin{thm}\label{thm:me-embedding}
Let $\Omega$ be an ME coupling of countable groups $G$ and $H$ with associated fundamental domains $X = \Omega//H$, $Y=\Omega//G$ and cocycles $\alpha:G \times X \to H, \beta: H \times Y \to G$ (as in section 2.1). Then $\alpha$ and $\beta$ are measurable embeddings. In fact the constant $C>0$ in Definition \ref{defn:embedding} can be taken to be $\lceil \frac{m(X)}{m(Y)} \rceil$, the least integer greater than or equal to $\frac{m(X)}{m(Y)}$.
\end{thm}

\begin{proof}

By symmetry, it suffices to show $\alpha$ is a measurable embedding. By decomposing $\Omega$ into ergodic components, we may assume without loss of generality that $G \times H \cc \Omega$ is ergodic. Therefore, there exists a measurable map $\phi:X \to G\times H$  such that if $\psi:X \to \Omega$ is defined by $\psi(x)=\phi(x)x$ then $\psi$ is at most $\lceil \frac{m(X)}{m(Y)}  \rceil$-to-1 and the image of $\psi$ lies in $Y$. Let $\pi_H: G\times H \to H$ be the projection map and define $\alpha':G\times X \to H$ by
$$\alpha'(g,x) = \pi_H(\phi(T^g x)) \alpha(g,x) \pi_H(\phi(x))^{-1}.$$
We claim that $\alpha'_x$ is at most $\lceil \frac{m(X)}{m(Y)} \rceil$-to-1 for a.e. $x\in X$. To see this suppose $g_1,\ldots, g_n \in G$ are distinct elements and $\alpha'(g_i,x)=\alpha'(g_j,x)$ for $1\le i,j \le n$. Then
\begin{eqnarray}\label{eqn:1}
\pi_H(\phi( T^{g_i}  x)) \alpha(g_i,x) =  \pi_H(\phi( T^{g_j} x)) \alpha(g_j,x)\quad 1\le i,j \le n.
\end{eqnarray}
Define $\Phi:\Omega \to Y$ by $\Phi(x) = gx$ where $g\in G$ is the unique element with $gx \in Y$. Note that $\Phi$ is $G$-invariant. Let $\psi':X \to \Omega$ be the map $\psi'(x) = \pi_H(\phi(x)) x$ so that $\psi(x) = \Phi(\psi'(x))$. Then for any $j$
\begin{eqnarray*}
\Phi(\pi_H(\phi( T^{g_j} x)) \alpha(g_j,x) x) &=& \Phi(g_j \pi_H(\phi( T^{g_j} x)) \alpha(g_j,x) x)\\ &=& \Phi(\pi_H(\phi( T^{g_j} x)) \alpha(g_j,x) g_j x) \\
&=& \Phi(\pi_H(\phi( T^{g_j} x)) T^{g_j} x) = \Phi(\psi'(T^{g_j} x)) = \psi(T^{g_j} x).
\end{eqnarray*}
Since $\Phi(\pi_H(\phi(T^{g_i} x)) \alpha(g_i,x) x) =\Phi(\pi_H(\phi(T^{g_j} x)) \alpha(g_j,x) x)$, this implies 
$$\Phi(\psi'(T^{g_i} x)) = \psi(T^{g_i} x) = \psi(T^{g_j} x)= \Phi(\psi'(T^{g_j} x)).$$

{\noindent} \textbf{Claim}. For any $i\ne j$, $\psi'(T^{g_i} x)\ne \psi'(T^{g_j} x)$. 

This claim implies that if $i\ne j$ then $T^{g_i} x \ne T^{g_j} x$. Because $\psi$ is at most $\lceil \frac{m(X)}{m(Y)} \rceil$-to-1, this implies that $n \le \lceil \frac{m(X)}{m(Y)} \rceil$. So it suffices to prove the claim.

So suppose that $\psi'(T^{g_i} x)=\psi'(T^{g_j} x)$. Then 
\begin{eqnarray*}
\pi_H(\phi(T^{g_i} x)) \alpha(g_i,x)g_i x &=& \pi_H(\phi(T^{g_i} x)) T^{g_i} x  = \psi'(T^{g_i} x) = \psi'(T^{g_j} x)\\
 &=& \pi_H(\phi(T^{g_j} x)) T^{g_j} x = \pi_H(\phi(T^{g_j} x)) \alpha(g_j,x)g_j x.
\end{eqnarray*}
By (\ref{eqn:1}) this implies $g_ix=g_jx$ which implies $g_i=g_j$ since $G \cc \Omega$ is essentially free. But $g_i$ and $g_j$ are distinct unless $i=j$. This proves the claim and the theorem. 
\end{proof}

\subsection{Integrable embeddings}

The definition of integrable embedding is a bit more complicated than measurable embedding because we only require that $\alpha_x$ is bounded-to-1 for a large subset of $x$ and with $\alpha_x$ is restricted to the associated return time set.

\begin{dfn}
For $Z \subset X$ and $x\in X$, $R_Z(x):=\{g\in G:~ gx \in Z\}$ is the associated {\em return time set}. 
\end{dfn}

\begin{dfn}
Let $G \cc (X,\mu)$ be an fmp action and $\alpha:G \times X \to H$ a measurable cocycle. Then $\alpha$ is an {\em integrable embedding} if for $\epsilon>0$ there exists a cocycle $\alpha':G \times X \to H$ which is cohomologous to $\alpha$ such that
\begin{itemize}
\item $\alpha'$ is integrable;
\item there exists a subset $X_0 \subset X$ with $\mu(X_0)>\mu(X)-\epsilon$ and a constant $C=C(\epsilon)>0$ such that for a.e. $x\in X_0$ and every $h\in H$, 
$$|\{g\in R_{X_0}(x):~ \alpha'(g,x)=h\}| \le C.$$
\end{itemize}
\end{dfn}


\begin{thm}\label{thm:lp-embedding}
Let $\Omega$ be an IME coupling of $G$ and $H$ with associated fundamental domains $X = \Omega//H$ and $Y=\Omega//G$ and cocycles $\alpha:G \times X \to H, \beta: H \times Y \to G$. Then $\alpha$ and $\beta$ are integrable embeddings.
\end{thm}

\begin{proof}
By symmetry, it suffices to show that $\alpha$ is an integrable embedding. By Theorem \ref{thm:me-embedding} there exists a cocycle $\alpha':G \times X \to H$ and a constant $C>0$ such that that $\alpha'$ cohomologous to $\alpha$ and $\alpha'_x$ is at most $C$-to-1 for a.e. $x\in X$. Because $\alpha'$ is cohomologous to $\alpha$, there exists a measurable map $\phi:X \to H$ such that
$$\alpha'(g,x) = \phi( T^{g} x)\alpha(g,x)\phi(x)^{-1}.$$
Let $S_G$ be a finite generating set for $G$. Choose a finite subset $W \subset H$ large enough such that if
$$X_0=\{x \in X:~ \phi(x) \in W \textrm{ and } \phi(T^{g} x) \in W ~\forall g\in S_G\}$$
then  $m(X_0) > m(X) - \epsilon$. Define $\phi_0:X \to H$ by
\begin{displaymath}
\phi_0(x)  = \left\{ \begin{array}{cc}
\phi(x) & \textrm{ if } x\in X_0 \\
e_H & \textrm{ otherwise }
\end{array}\right.
\end{displaymath}
Define $\alpha'': G\times X \to H$ by
$$\alpha''(g,x) = \phi_0(T^{g} x) \alpha(g,x) \phi_0(x)^{-1}.$$
For a.e. $x\in X_0$, $\alpha''_x$ restricted to $R_{X_0}(x)$ equals $\alpha'_x$ and is therefore at most $C$-to-1. 

To finish the proof it suffices to show that $\alpha''$ is an integrable cocycle.  Let $M= \max_{h \in W} |h|_H$ and $g\in S_G$. Then
\begin{eqnarray*}
\int |\alpha''(g,x)|_H~d\mu_X(x) &\le& \int (2M+|\alpha(g,x)|_H)~d\mu_X(x) < \infty
\end{eqnarray*}
because $\alpha$ is integrable. It now follows from sub-additivity that $\int |\alpha''(g,x)|_H~d\mu_X(x)  < \infty$ for every $g\in G$.


\end{proof}

\subsection{Growth}


Our main result is:

\begin{thm}\label{thm:growth}
Let $G$, $H$ be two finitely generated groups. If there exists an $L^1$-embedding of $G$ into $H$ then $\gr_G \lesssim \gr_H$. 
\end{thm}

This result and Theorem \ref{thm:lp-embedding} immediately imply Theorem \ref{thm:growth2}. To prove Theorem \ref{thm:growth} we need the following simple lemma:
 \begin{lem}\label{lem:simple}
Let $G$ be a finitely generated group. Let $G \cc (X,\mu)$ be an fmp action,  $X_0 \subset X$ a set with positive measure and $R_{X_0}(x) :=\{g\in G:~gx \in X_0\}$ the associated return time set. If $B_G(e,n)$ denotes the ball of radius $n$ centered at the identity in $G$ (with respect to a fixed word metric) then for every $n$
$$ \int_{X_0} \frac{|R_{X_0}(x) \cap B_G(e,n)|}{|B_G(e,n)|} ~d\mu(x) \ge 2\mu(X_0) - \mu(X).$$
\end{lem}

\begin{proof}
By integrating over $X$ in place of $X_0$ and using that $G\cc X$ is measure-preserving we see that
\begin{eqnarray*}
\mu(X_0)= \int_{X} \frac{|R_{X_0}(x) \cap B_G(e,n)|}{|B_G(e,n)|} ~d\mu(x).
\end{eqnarray*}
Therefore 
\begin{eqnarray*}
\int_{X_0} \frac{|R_{X_0}(x) \cap B_G(e,n)|}{|B_G(e,n)|} ~d\mu(x) &=& \mu(X_0) - \int_{X\setminus X_0} \frac{|R_{X_0}(x) \cap B_G(e,n)|}{|B_G(e,n)|} ~d\mu(x)\\ 
&\ge& \mu(X_0) - \mu(X \setminus X_0) = 2\mu(X_0) - \mu(X).
\end{eqnarray*}

\end{proof} 

\begin{proof}[Proof of Theorem \ref{thm:growth}]
Let $G \cc (X,\mu)$ be an ess. free fmp action and $\alpha:G \times X \to H$ an $L^1$-embedding. After replacing $\alpha$ with a cohomologous cocycle if necessary we may assume there exists a set $X_0 \subset X$ with $\mu(X_0) \ge 0.9\mu(X)$ and a constant $C>0$ such that for a.e. $x\in X_0$, $\alpha_x$ restricted to the return time set $R_{X_0}(x)$ is at most $C$-to-1 (where $\alpha_x:G \to H$ is defined by $\alpha_x(g)=\alpha(g,x)$).

For $g\in G$, let $\kappa(g) := \int |\alpha(g,x)|_G ~d\mu(x)$. An easy exercise shows $\kappa(gh) \le \kappa(g) + \kappa(h)$. Let $M=\sup_{g\in S} \kappa(g)$ (where $S\subset G$ is the finite symmetric generating set defining the word metric). Let $B_G(e,n)$ denote the ball of radius $n$ in $G$. Note that
$$\sum_{g\in B_G(e,n)} \kappa(g) \le \sum_{g \in B_G(e,n)} |g|_GM.$$
By Markov's inequality, for any $t>0$,
$$\mu\left(\left\{ x\in X:~ \sum_{g\in B_G(e,n)} \frac{|\alpha(g,x)|_H}{|g|_G} \ge t \right\}\right) \le  \frac{1}{t} \sum_{g\in B_G(e,n)} \frac{\kappa(g)}{|g|_G}   \le \frac{|B_G(e,n)|  M}{t}.$$
In particular, by setting $t= 10  M |B_G(e,n)|$, we have that $\mu(X_1)\ge 0.9\mu(X)$ where 
$$X_1:=\left\{ x\in X:~ \sum_{g\in B_G(e,n)} \frac{|\alpha(g,x)|_H}{|g|_G} < 10M |B_G(e,n)| \right\}.$$
If $x\in X_1$ and
$$G_x=\{g\in B_G(e,n):~|\alpha(g,x)|_G \le 60 M |g|_G\}$$
then $|G_x| \ge (5/6) |B_G(e,n)|$. Let $X_2=X_0\cap X_1$. Observe that 
$$\mu(X_2)=\mu(X_0)+\mu(X_1) - \mu(X_0 \cup X_1) \ge \mu(X_0)+\mu(X_1) - \mu(X) \ge 0.8\mu(X).$$
By Lemma \ref{lem:simple}, 
$$\int_{X_2} \frac{|R_{X_2}(x) \cap B_G(e,n)|}{|B_G(e,n)|} ~d\mu(x) \ge 2\mu(X_2) - \mu(X) \ge 0.6\mu(X).$$
So
\begin{eqnarray*}
&&\int_{X_2} |R_{X_2}(x) \cap G_x|  ~d\mu(x)\\ &&\ge \int_{X_2} |R_{X_2}(x) \cap B_G(e,n)| + |G_x| - |B_G(e,n)|~d\mu(x) \\
&&\ge 0.6\mu(X)|B_G(e,n)| + (5/6)|B_G(e,n)|\mu(X_2)  - |B_G(e,n)|\mu(X_2) \\
&&\ge |B_G(e,n)| \mu(X) ( 0.6 + 0.5 - 1) = 0.1\mu(X) |B_G(e,n)|.
\end{eqnarray*} 
On the other hand, for every $x\in X_2$, $\alpha_x$ restricted to $R_{X_2}(x)$ is at most $C$-to-1. Therefore 
\begin{eqnarray*}
 0.1\mu(X) |B_G(e,n)| &\le& \int_{X_2} |G_x \cap R_{X_2}(x) |~d\mu(x)\\ &\le& \int_{X_2} |\{g \in R_{X_2}(x):~|\alpha(g,x)|_H \le 60Mn\}| ~d\mu(x) \\
&\le& \sum_{h \in B_H(e,60Mn)} \int_{X_2} |\{g \in R_{X_2}(x):~ \alpha(g,x) =h \}| ~d\mu(x) \\
&\le& C |B_H(e,60Mn)| \mu(X).
\end{eqnarray*}
So $|B_G(e,n)| \le C |B_H(e,60Mn)|$. Since this is true for all $n$, $\gr_G \lesssim \gr_H$.
\end{proof}


\bibliographystyle{abbrv}
\bibliography{bibfile}

\noindent\small{T. Austin, Courant Institute of Mathematical Sciences, New York University, 251 Mercer Street, New York, NY 10012, U.S.A.}

\noindent\small{Email: \texttt{tim@cims.nyu.edu}}

\vspace{7pt}

\noindent\small{L. Bowen, University of Texas at Austin, Mathematics Department, RLM 8.100, 2515 Speedway Stop C1200,
Austin, TX 78712, U.S.A.}

\noindent\small{Email: \texttt{lpbowen@math.utexas.edu}}

\end{document}